\documentclass[12pt]{article}

\usepackage{amsmath}
\usepackage{amssymb}
\usepackage{amsthm}
\usepackage{amsfonts}
\usepackage{amscd}
\usepackage{graphicx}
\usepackage[numbers,sort&compress]{natbib}
\usepackage{longtable}
\usepackage{hyperref}
\usepackage{makeidx}
\usepackage[nottoc]{tocbibind}  
\usepackage[usenames,dvipsnames]{color}
\usepackage{textcomp}
\usepackage{mathrsfs}
\usepackage{array}
\usepackage{marvosym}
\usepackage{enumitem} 

\usepackage{setspace}
\hypersetup{
    colorlinks=true,
    breaklinks=true,
    filecolor=blue,
    urlcolor=Violet,
    linkcolor=blue,
    citecolor=Green,
    anchorcolor=blue,
    frenchlinks=true,
    pdfborder={0 0 0},
    pdfpagelayout=TwoPageRight,
    pdfdisplaydoctitle=true,
    bookmarksnumbered=true
}

\theoremstyle{plain}

\newtheorem{theorem}{Theorem}

\theoremstyle{definition}

\newtheorem{remark}{Remark}
\newtheorem{example}{Example}
\newtheorem{algorithm}{Algorithm}

\newcounter{figcnt}
\renewenvironment{figure}{\refstepcounter{figcnt}\begin{center}}{\end{center}}

\makeindex
    
    \newcommand{\Mend}{\hfill \ensuremath{\vartriangleleft}}

    \newcommand{\Msum}[2]{\ensuremath{\overset{#2}{\underset{#1}{\sum}}}}

    \newcommand{\Mmod}[1]{\langle #1\rangle} 

    \newcommand{\Mset}[2]{\ensuremath{\{~ #1 ~|~ #2 ~\}}}

    \newcommand{\Miff}{if and only if }

    \newcommand{\Mst}{{such that }}
    
    \newcommand{\Mresp}{respectively}
    \newcommand{\Mwrt}{with respect to }

    \newcommand{\Marrow}[3]{\ensuremath{#1\stackrel{#2}{\longrightarrow}#3}}
    
    \newcommand{\Mdasharrow}[3]{\ensuremath{#1\stackrel{#2}{\dashrightarrow}#3}}
    
    \newcommand{\Mrow}[3]{\ensuremath{#1\colon #2\longrightarrow #3}}
    
    \newcommand{\Mdashrow}[3]{\ensuremath{#1\colon #2\dashrightarrow #3}}

    \newcommand{\Mfig}[3]{{\includegraphics[height=#1cm, width=#2cm]{img/#3}}}

    \newcommand{\Mdef}[1]{\textit{#1}\index{#1}}
    \newcommand{\MdefAttr}[2]{\textit{#1}\index{#2!#1}\index{#1}}

    \newcommand{\Me}{e}
    \newcommand{\Mk}{k}
    
    \newcommand{\Mh}{h}

    \newcommand{\McalA}{{\mathcal{A}}}\newcommand{\McalB}{{\mathcal{B}}}\newcommand{\McalG}{{\mathcal{G}}}\newcommand{\McalH}{{\mathcal{H}}}\newcommand{\McalP}{{\mathcal{P}}}\newcommand{\McalR}{{\mathcal{R}}}\newcommand{\McalS}{{\mathcal{S}}}
    \newcommand{\MbbP}{{\mathbb{P}}}\newcommand{\MbbR}{{\mathbb{R}}}\newcommand{\MbbS}{{\mathbb{S}}}\newcommand{\MbbZ}{{\mathbb{Z}}}

    \newcommand{\EQN}[1]{{(\ref{eqn:#1})}}
    \newcommand{\SEC}[1]{{\textsection\ref{sec:#1}}}
    \newcommand{\FIG}[1]{{Figure~\ref{fig:#1}}}
    \newcommand{\ALG}[1]{{Algorithm~\ref{alg:#1}}}
    
    \newcommand{\EXM}[1]{{Example~\ref{exm:#1}}}

\title{Computing curves on real rational surfaces}

\author{Niels Lubbes}

\date{\today}

\begin{document}

\maketitle

\begin{abstract}
We present an algorithm for computing curves and families of curves 
of prescribed degree and geometric genus on real rational surfaces.


\end{abstract}

\section{Introduction}

Suppose we are given a birational map $\Mdashrow{\McalH}{\MbbP^2}{X\subset\MbbP^n}$.
\ALG{fam} computes all families of curves on $X$ of degree $\alpha$, 
projective dimension $\nu-1$ and geometric genus $\rho$.
If $\nu=1$, then each family consists of a single curve. 
We use the basepoint analyis method \citep[Algorithm~1~and~2]{nls-bp} to reduce 
the problem of computing curves in a surface
to the problem 
of computing elements in a unimodular lattice.
In order to clarify the input/output specification of \ALG{fam} we 
list some classical usecase examples that we can compute:
\begin{itemize}[topsep=0pt,itemsep=0pt]

\item Both 1-dimensional family of lines on a hyperboloid of one sheet.
These families were discovered by \citep[Christopher Wren, 1669]{wrn1}.

\item The 27 real and/or complex conjugated lines in a smooth cubic surface.
Cayley-Salmon theorem states the existence of these lines \citep[1848]{cay1}.

\item Four 1-dimensional families of circles on a ring torus.
Two of these families are families of Villarceau circles \citep[1848]{vil1}.

\item The 2-dimensional family of conics on the Roman surface.
This surface was discovered by \citep[Steiner, 1844]{stei} in Rome
and is a projection of the Veronese surface (\FIG{1}).

\item The four circles contained in the Roman surface (\FIG{1}).
\end{itemize}
See \citep[\texttt{linear\_series}]{linear_series} and \citep[\texttt{ns\_lattice}]{ns_lattice} for an
implementation of algorithms used in this paper.
\begin{figure}
\label{fig:1}
\begin{center}
\Mfig{4}{4.5}{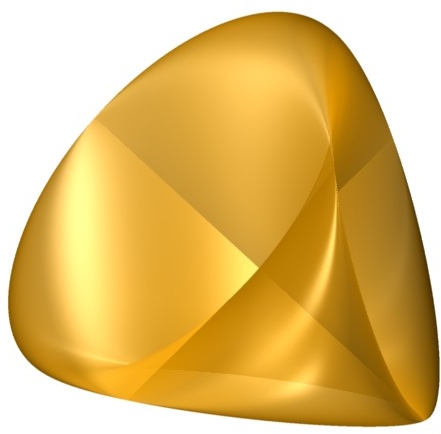}\qquad\qquad
\Mfig{4}{4.5}{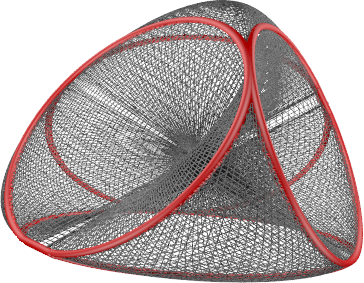} 
\end{center}
\textbf{Figure 1.}
{\it%
A web of conics and the four circles on the Roman surface.
}
\end{figure}
In \cite{diaz,sen2} algorithms are proposed to compute 
families of lines on $X$, in case $X$ is geometrically ruled.
Such families are represented in terms of reparametrizations of $\McalH$
and it is shown that basepoints of $\McalH$ can be moved to the line at infinity.
In \cite{alc1} an algorithm is proposed to compute straight lines in a rational surface,
using methods from differential geometry.
The more general algorithm proposed in his article, has the advantage 
we can also compute curves that are not reached by the parametrization $\McalH$ 
\citep[Section~4.1]{nls-bp}. 
However, as the parametric degree of $\Mdashrow{\McalH}{\MbbP^2}{X\subset\MbbP^n}$
and embedding dimension $n$ is increasing, our algorithm is less likely to terminate.
The bottleneck is the basepoint analysis step \citep[Remark~1]{nls-bp}.
It is likely that there exists rational surfaces where the methods of \cite{alc1,diaz,sen2} are favourable,
when computing straight lines contained in these surfaces.

We explain the structure of this paper by summarizing the main steps of \ALG{fam}.
After introducing basic notions in \SEC{pre},
we analyze in \SEC{comp-ns} the basepoints of the linear series corresponding to the map $\McalH$.
From these basepoints we recover generators for the Neron-Severi lattice $N(X)$.
In \SEC{comp-class} and \SEC{alg} we compute with \ALG{fam-class}
a set of candidate divisor classes $\McalA\subset N(X)$ using input $\alpha$ and $\rho$.
For each class $c\in \McalA$ we compute a linear series $|c|$ in the plane.
The linear series $|c|$ is valid if it is of the requested projective dimension $\nu-1$ 
and if the general curve in the linear series is irreducible. 
The curve $\McalH(C)\subset X$ 
is of given degree and geometric genus for all curves $C$ in a valid linear series $|c|$.
In \SEC{exm-alg} we explain \ALG{fam} with an example.
In \SEC{circles} we compute circles in rational surfaces, by
composing the birational map $\McalH$ with an inverse stereographic projection.

\section{Preliminaries}
\label{sec:pre}

\subsection{Real varieties}
\label{sec:real}

We define a \Mdef{real variety} $X$ as a complex variety together
with an antiholomorphic involution $\Mrow{\sigma}{X}{X}$, called the \Mdef{real structure}.
We implicitly assume that all structures are compatible with the 
real structure. For example, the birational map $\Mdashrow{\McalP}{\MbbP^2}{X}$ is real
unless explicitly stated otherwise.
The \Mdef{smooth model} of a surface~$X$ is a birational morphism $\Marrow{Y}{}{X}$ 
from a nonsingular surface~$Y$, that does not contract exceptional curves.

\subsection{Neron-Severi lattice}
\label{sec:ns}

For computational purposes, we make the data associated 
the well-known Neron-Severi lattice explicit \citep[page 461]{jha2}.
The \Mdef{Neron-Severi lattice} $N(X)$ of a rational surface~$X\subset \MbbP^n$ 
    consists of the following data:
\begin{enumerate}[topsep=0pt,itemsep=0pt]
\item
A unimodular lattice defined by divisor classes on the smooth model~$Y$ of $X$, modulo numerical equivalence.

\item
A basis for the lattice.   
We say that a basis is of \Mdef{type 1}
if the generators are 
$\Mmod{\Me_0,\Me_1,\ldots,\Me_r}$
\Mst 
the nonzero intersections are $e_0^2=1$ and $e_i^2=-1$ for $1\leq i\leq r$.
We assume a basis of type 1 unless explicitly stated otherwise.

\item
A unimodular involution $\Mrow{\sigma_*}{N(X)}{N(X)}$ induced by the real structure of~$X$.

\item
A function $\Mrow{h^0}{N(X)}{\MbbZ_{\geq0}}$ assigning the dimension of global sections
of the line bundle associated to a class.

\item
Two distinguished elements $\Mh,\Mk\in N(X)$ corresponding to
class of a hyperplane sections and the canonical class respectively.
\end{enumerate}

\subsection{Complete families}
\label{sec:fam}

We assume that $X\subset\MbbP^n$ is a rational surface 
\Mst linear-, algebraic- and numerical- equivalence of
divisors on~$X$ coincide. 

A \Mdef{family} is defined as a divisor $F\subset X\times\MbbP^m$ 
\Mst the second projection $\Mrow{\pi_2}{F}{\MbbP^m}$ is dominant.
We say that $F$ \Mdef{covers} $X$ if the first projection $\Mrow{\pi_1}{F}{X}$ is dominant as well.
A \Mdef{member} of $F$, 
corresponding to $i\in \MbbP^m$, is defined as the curve $F_i:=(\pi_1\circ\pi_2^{-1})(i)\subset X$.

We can associate to a curve $C\subset X$ its class $[C]\in N(X)$ \Mst
classes $[C]$ and $[C']$ are equal \Miff $C$ and $C'$ are members of the same family.
The class $[F]$ of $F$ is defined as the class of any of its members.

We call a family $F$ \MdefAttr{complete}{family} if there exists a curve $C\subset X$
\Mst the set $\Mset{C'\subset X}{ [C']=[C] }$ defines exactly the set of members of $F$.
A family is called \MdefAttr{irreducible}{family} if its general member is an irreducible curve.
The \MdefAttr{dimension}{family} of $F$ is defined as $m$.
A 0-dimensional family consists of a single curve.
If $F$ is complete then $h^0([F])=m+1$.
The \Mdef{arithmetic genus} $p_a([F])$ of a complete irreducible family $F$ is defined as the geometric genus 
of its general member. The \Mdef{degree} of $F$ is defined as the degree of any member 
\Mwrt the embedding $X\subset\MbbP^n$.

\section{Computing the Neron-Severi lattice}
\label{sec:comp-ns}

In this section
we explain in terms of an example, how to compute the Neron-Severi lattice of a rational surface.

We consider the following birational map:
\begin{equation}
\label{eqn:H}
\begin{array}{rcl}
             \MbbP^2       & \Mdasharrow{}{\McalH}{} & X\subset\MbbP^4 \\
             (x_0:x_1:x_2) & \longmapsto     & 
(
x_1^3+2x_2^2x_0-x_1x_0^2-2x_2x_0^2: 
x_1^2x_2:\\&& 
x_1x_2^2-2x_2^2x_0+2x_2x_0^2:
x_1x_2x_0-2x_2^2x_0+2x_2x_0^2:\\&&
x_2^3-2x_2^2x_0+x_2x_0^2
).
\end{array}
\end{equation}
Let $|\Mh|$ be the linear series associated to the birational map $\McalH$
so that $\Mh$ is the class of hyperplane sections of $X$.
With \citep[Algorithm 1]{nls-bp} we find the basepoints of $\McalH$ in the affine chart~$U_0\subset\MbbP^2$ 
defined by $x_0\neq 0$:
\[
\Gamma=\{~
\Mmod{((),p_1,1)},~
\Mmod{((),p_2,1)},~
\Mmod{((),p_3,1)},~
\Mmod{(),p_4,1},~
\Mmod{(p_5,t),p_2}~
\} 
\]
where $p_1=(-1,0)$, $p_2=(0,0)$, $p_3=(1,0)$, $p_4=(0,1)$ and $p_5=(2,0)$.
There are no basepoints outside $U_0$.
Notice that basepoints $p_1$, $p_2$ and $p_3$ are collinear simple basepoints 
corresponding to $(-1:0:1)$, $(0:0:1)$ and $(1:0:1)$ in~$\MbbP^2$ \Mresp.
Basepoint $p_4$ has projective coordinates $(0:1:1)$ and basepoint $p_5$
is infinitely near to $p_4$.

The Neron-Severi lattice of~$X$ has a basis of type 1:
\[
N(X)=\Mmod{\Me_0,\Me_1,\Me_2,\Me_3,\Me_4,\Me_5},
\]
The induced real structure
$\Mrow{\sigma_*}{N(X)}{N(X)}$ is the identity map, since all basepoints are real.
Moreover, 
\[
\Mh=-\Mk=3\Me_0-\Me_1-\Me_2-\Me_3-\Me_4-\Me_5,
\]
since $\Mh$ consists of cubic polynomials that pass with multiplicity one
through each of the five basepoints.
The smooth model~$Z$ of~$X$ is isomorphic to the projective plane blownup in the 
base locus of~$\McalH$:
\begin{equation*}
\begin{array}{r@{}c@{}l}
        &           Z                  &     \\
        &\tau_1\swarrow~~\searrow\tau_2 &     \\
\MbbP^2 & \Mdasharrow{}{\McalH}{}      & X          
\end{array}
\end{equation*}
%
%
The generator $\Me_0$ is the class of the pullback along $\tau_1$, of lines in the plane 
and $\Me_i$ for $1\leq i\leq 5$ is the class of the pullback of the exceptional curve
resulting from the blowup of basepoint $p_i$.
The surface $Z$ is in general not a smooth model of $X$,
since $\tau_2$ may contract exceptional curves.
See \citep[Section 4.3]{nls-bp} for more details on computing the Neron-Severi lattice
of a rational surface.

\section{Computing families from their classes }
\label{sec:comp-class}

Suppose that $\Mdashrow{\McalH}{\MbbP^2}{X}$ is given by \EQN{H}
and that $F\subset X\times\MbbP^m$ is a complete family.
In this section we explain in terms of examples how to compute from the class $[F]\in N(X)$
the following attributes as defined in \SEC{fam} (see also \citep[Section 4.4]{nls-bp}):
\begin{itemize}[topsep=0pt, itemsep=0pt]
\item A complete linear series $L$ so that $\McalH(C)\subset X$ is a member of $F$ for all curves $C$ in $L$.
By abuse of notation we shall denote $[\McalH(C)]=[F]$ in $N(X)$ by $[L]$.
\item Dimension $h^0([F])-1$, arithmetic genus $p_a([F])$, and degree $\deg F$.
\item Reducibility of $F$.
\end{itemize}
We computed in \SEC{comp-ns} the Neron-Severi lattice $N(X)$ and
we know that 
\[
[F]=c=c_0\Me_0+c_1\Me_1+\ldots+c_5\Me_5 \in N(X),
\]
for some $c_i\in \MbbZ$ for $0\leq i\leq 5$.
Now there are two possibilities:
\begin{enumerate}[topsep=0pt, itemsep=0pt]
\item 
$c_0>0$ and $c_i\leq 0$ for $i>0$. The members of the family $F$
are in this case, via $\tau_{2*}\circ \tau_1^*$, strict transforms of curves in $\MbbP^2$ of degree $c_0$
that pass through basepoints $p_i$ with multiplicity $-c_i$.

\item 
$[F]=\Me_i$ for some $i>0$. In this case $F$ is a $0$-dimensional family
whose unique member is $\tau_2(E)\subset X$, where $E\subset Z$ is an exceptional curve
\Mst $\tau_1(E)=p_i\in\MbbP^2$.

\end{enumerate}
We remark that
if $p_i$ is infinitely near to $p_j$ then 
there is exists a curve in $Z$ with class $\Me_i-\Me_j$.
This \Mdef{(-2)-curve} is contracted to an isolated singularity by $\Mrow{\tau_2}{Z}{X}$.

We will now consider several scenarios, where we determine attributes of $F$
when only its class $[F]\in N(X)$ is given. 

\begin{itemize}[topsep=0pt, itemsep=0pt]

\item 
Suppose that $[F]=\Me_0-\Me_4-\Me_5$.
Recall that \citep[Algorithm 2]{nls-bp} 
determines which curves in a given linear series, 
form a linear series with prescribed basepoints.
The following linear series is defined by
a monomial basis for linear polynomials in~$\MbbR[u,v]$:
$G:=(u,v,1)$.
The line in $G$ that passes through basepoints $p_4$ and $p_5$
is after homogenization $L_{45}:=\Mset{x\in\MbbP^2}{ x_0-2x_1+2x_2=0 }$ \Mst $u=\frac{x_1}{x_0}$ and $v=\frac{x_2}{x_0}$.
Thus $[F]=[L_{45}]$ and 
the image~$\McalH(L_{45})\subset X$ is the 
unique member of~$F$. 
The linear series for $L_{45}$ consists of a single curve so that $F$ is of dimension $h^0([L_{45}])-1=0$.
We verify from the arithmetic genus formula that $2p_a([L_{45}])-2=[L_{45}]^2+[L_{45}]\cdot\Mk=-2$.
The degree of the image~$\McalH(L_{45})\subset \MbbP^4$ 
equals $\Mh\cdot [L_{45}]=1$, thus the unique member of $F$ is a line in $X$.

\item Suppose that $[F]=\Me_0-\Me_4$.
We compute with \citep[Algorithm 2]{nls-bp} the 
linear subseries of lines in $G$ that pass through $p_4$.
After projectivization this linear series is defined by the following tuple 
of linear forms $L_4:=(x_1,x_2-x_0)$ so that $h^0([L_4])=2$ and $[F]=[L_4]$. 
Let 
$C_\alpha:=\Mset{ x\in\MbbP^2 }{ \alpha_0x_1+\alpha_1(x_2-x_0)=0 }$
denote a line in the linear series of $L_4$ for all $\alpha=(\alpha_0:\alpha_1)\in\MbbP^1$. 
Since $\Mh\cdot [L_4]=2$ it follows that the image  $\McalH(C_\alpha)\subset X$
is a conic for all $\alpha\in\MbbP^1$. 
Indeed, $F$ defines a 1-dimensional complete family of conics on $X$.
We consider the following parametrization of $C_\alpha$:
\begin{equation*}
\begin{array}{rcl}
\MbbP^1\times\MbbP^1        & \Mdasharrow{}{\McalG}{} & \MbbP^2 \\
(\alpha_0:\alpha_1;t_0:t_1) & \longmapsto             & 
(
\alpha_0t_0:-\alpha_1t_1:\alpha_0t_1+\alpha_0t_0
).
\end{array}
\end{equation*}
Notice that $\Mdashrow{\McalH\circ\McalG}{\MbbP^1\times\MbbP^1}{X}$ 
is a reparametrization of $X$ \Mst if 
we fix $\alpha\in\MbbP^1$ then we obtain a parametrization of 
the conic $\McalH(C_\alpha)\subset X$.

\item 
Suppose that $[F]=\Me_0-\Me_1-\Me_2-\Me_3$.
The curve in $G$, through basepoints $p_1$, $p_2$ and $p_3$,
is after projectivization the line~$L_{123}:=\Mset{x\in\MbbP^2}{x_2=0}$
so that $h^0([L_{123}])=1$. However, $\Mh\cdot [L_{123}]=0$ 
so that $L_{123}$ is contracted by $\McalH$ to an isolated singularity.
In particular, $F$ does in this case not define a family on $X$.
Since $[L_4]\cdot[L_{123}]=1$ it follows that each member of 
the family of conics with class $[L_4]$, passes through this isolated singularity.

\item
Suppose that $[F]=\Me_0-\Me_1-\Me_5$.
A necessary condition for $c\in N(X)$ being the class of a line 
in $X$, is that $\Mh\cdot c=1$. However, this condition is not sufficient. 
In this case we have $[F]=(\Me_0-\Me_1-\Me_4)+(\Me_4-\Me_5)$ so that $\Mh\cdot [F]=1$, 
but $[F]$ cannot be the class of a line through $p_1$ and an infinitely near point $p_5$, 
since such a line also passes through $p_4$.
We can detect in $N(X)$ that $[F]$ is not the class of a line,
since $[F]\cdot (\Me_4-\Me_5)<0$ and $h^0(\Me_4-\Me_5)>0$ so that $\Me_4-\Me_5$
is the class of a (-2)-curve.
In particular, we find that $F$ is not irreducible.
Notice that the curve with class $\Me_4-\Me_5$ in the smooth model~$Z$
is contracted by $\tau_2$ to an isolated singularity of~$X$. 
The line $\McalH(L_{14})\subset X$ with class $\Me_0-\Me_1-\Me_4$ passes 
through this singular point, since $(\Me_0-\Me_1-\Me_4)\cdot (\Me_4-\Me_5)=1$.

\item 
Suppose that $[F]=2\Me_0-\Me_1-\Me_2-\Me_3-\Me_4-\Me_5$.
In this case, $[F]$ is the class of the pullback of a conic 
$C\subset \MbbP^2$ along $\Mrow{\tau_1}{Z}{\MbbP^2}$. 
The conic $C$ passes through $p_1$, $p_2$, $p_3$ and $p_4$
\Mst its tangent direction at $p_4$ is determined by the infinitely near point 
$p_5$. We notice that $[F]\cdot (\Mh-\Me_1-\Me_2-\Me_3)<0$ and therefore 
the line $L_{123}$ through $p_1$, $p_2$ and $p_3$ must be a
component of $C$. Indeed, $[F]=[L_{45}]+[L_{123}]$. 
We compute with \citep[Algorithm 2]{nls-bp}
the curves in the linear series of all conics 
that pass through $p_1,\dots, p_5$ and verify that
there is a unique reducible conic with equation $x_2(x_0-2x_1+2x_2)$. 
Thus $F$ consists of a reducible conic and is therefore a reducible family.

\item
Suppose that $[F]=2\Me_0-\Me_1-\Me_2-\Me_3-\Me_4$.
In this example, $\Mh\cdot [F]=2$.
We compute with \citep[Algorithm 2]{nls-bp}
the curves in the linear series of all conics 
that pass through $p_1,\dots, p_4$ and obtain
after projectivization the reducible linear series $\bigl(x_2x_1,x_2(x_2-x_0)\bigr)$.
We notice that $[F]\cdot [L_{123}]<0$ and thus $[L_{123}]$ is indeed a
fixed component of the linear series of conics passing through 
the simple basepoints $p_1$, $p_2$, $p_3$ and $p_4$.
We conclude that $F$ is a reducible family of conics.

\end{itemize}

\section{Algorithms}
\label{sec:alg}
We introduce in this section algorithms for computing complete irreducible families 
on rational surfaces of prescribed degree, dimension and arithmetic genus.

\begin{algorithm}
\textbf{(find elements in Neron-Severi lattice)}  
\label{alg:fam-class} 
\begin{itemize}[itemsep=0pt,topsep=0pt]

\item \textbf{Input:}
\begin{enumerate}[itemsep=0pt,topsep=0pt]
\item 
The class of hyperplane sections $\Mh=\Mh_0\Me_0+\ldots+\Mh_r\Me_r$ in a
Neron-Severi lattice $N(X)$
\Mwrt
type 1 basis $\Mmod{\Me_0,\ldots,\Me_r}$.

\item
Integers $\alpha,\beta \in\MbbZ_{\geq 0}$.
\end{enumerate}

\item \textbf{Output:}
\\
The set of classes $c\in N(X)$ \Mst 
$\Mh\cdot c=\alpha$, $c^2=\beta$
and either one of the following three conditions is satisfied:
\begin{enumerate}[itemsep=0pt,topsep=0pt]
\item $c=\Me_i-\Me_j$ for some $0<i,j\leq r$.
\item $c=\Me_i$ for some $i>0$.
\item $c\cdot \Me_i>0$ for all $0\leq i\leq r$.
\end{enumerate}

\item \textbf{Method:} {\it We use notation $c=c_0\Me_0+\ldots+c_r\Me_r$}.

\item[] $\McalS:=\Mset{ c\in N(X) }{ 
c\in\{\,\Me_i-\Me_j, \Me_i \,|\, 0<i,j\leq r \,\},\,
(\Mh\cdot c,c^2)=(\alpha,\beta)}$

\item[]
$f(t):=(h_0 t-\alpha)^2-(h_0^2 - h^2)(t^2 - \beta)$ 

\item[]
$c_0:=1$

\item[] 
\textbf{while} 
\quad$f(c_0)\leq 0$\quad 
\textbf{or} 
\quad$f(c_0)\leq f(c_0-1)$\quad 
\textbf{do}
\begin{itemize}[leftmargin=1cm,itemsep=0pt,topsep=0pt]

\item[]
Compute all $(c_i)_{i>0}\in\MbbZ^r_{\geq0}$ \Mst 
$\Mh_0c_0-\alpha=\Msum{i>0}{}\Mh_ic_i$ by going 
through integer partitions of $\Mh_0c_0-\alpha$.

\item[]  $\McalS:=\McalS\cup\Mset{c}{\Mh_0c_0-\alpha=\Msum{i>0}{}\Mh_ic_i,~ c^2=\beta}$

\item[]  $c_0:=c_0+1$
\end{itemize}

\item[] \textbf{return} $\McalS$
\Mend
\end{itemize}
\end{algorithm}

\begin{algorithm}
\textbf{(complete families on rational surfaces)} 
\label{alg:fam}
\begin{itemize}[itemsep=0pt,topsep=0pt]

\item \textbf{Input:} 
A birational map $\Mdashrow{\McalH}{\MbbP^2}{X\subset\MbbP^n}$
and integers $\alpha,\nu,\rho \in\MbbZ_{\geq 0}$.

\item \textbf{Output:}
Classes of complete irreducible families $F\subset X\times \MbbP^{\nu-1}$
with prescribed degree $\alpha$, dimension $\nu-1$ and geometric genus $\rho$:
\[
\Mset{ [F]\in N(X) }{ \deg F=\alpha,~ h^0([F])=\nu,~ p_a([F])=\rho,~ F \text{ is irreducible} }.
\]

\item\textbf{Method:}
\begin{enumerate}[topsep=0pt,itemsep=0pt]

\item
Compute $N(X)$ with type 1 basis $\Mmod{e_0,\ldots,e_r}$
by computing the base locus of $\McalH$ \citep[Algorithm 1 and Section 4.3]{nls-bp}
\Mst $\Mh \in N(X)$ is the class of the linear series associated to $\McalH$.

\item
We call \ALG{fam-class} with input $\Mh$, $\alpha$ and $\beta$,
where $\beta$ runs from $-2$ to  $\nu+\rho-2$. 
Let $\McalA$ be the union of the outputs for each $\beta$.
Each element in $\McalA$ is a vector \Mwrt the type 1 basis of $N(X)$.

\item 
We start with an empty set $\McalB$.
With \citep[Algorithm 2]{nls-bp} we compute
for each class in $\McalA$ its linear series $L$.
If $h^0([L])=\nu$ and $L$ is irreducible without unassigned basepoints in $\McalH$, 
then we set $\McalB:=\McalB\cup\{[L]\}$.
\end{enumerate}

\item[] \textbf{return} $\McalB$
\Mend
\end{itemize}
\end{algorithm}

Notice that we discussed step 3 of \ALG{fam} in \SEC{comp-class}.
We will see in \EXM{roman} a scenario
where the linear series has unassigned basepoints in $\McalH$.

\begin{theorem}
\textbf{(complete families on rational surfaces algorithm)}
\label{thm:fam}
\\
\ALG{fam-class} and \ALG{fam} are correct.
\end{theorem}

\begin{proof}
The halting of \ALG{fam-class} follows from the 
Cauchy-Schwarz inequality:
\[
(\Mh_0c_0-\alpha)^2=
\left(\Msum{i>0}{}\Mh_ic_i\right)^2 
\leq
\left(\Msum{i>0}{}\Mh_i^2\right)\left(\Msum{i>0}{}c_i^2\right)
=
(\Mh_0^2 - \Mh^2)(c_0^2 - \beta).  
\]
Thus we require that $f(c_0)\leq 0$. 
Notice that if $f(c_0)> f(c_0-1)$ and $f(c_0)>0$, then $f(t)>0$ for all $t\geq c_0$.   
The correctness of \ALG{fam-class} is now straightforward.

For the correctness of \ALG{fam} we recall that
$\alpha=\deg F=\Mh\cdot [F]$.
The arithmetic genus formula states that $-2\leq c^2+c\cdot \Mk=2p_a(c)-2$.
It follows from
Riemann-Roch theorem and Serre duality that 
$-2\leq c^2-c\cdot \Mk \leq 2h^0(c)-2$.
The sum of these equations implies that 
\[-2\leq c^2\leq h^0(c)+p_a(c)-2.\]
Thus $c^2=\beta\in[-2,\nu+\rho-2]$ and \ALG{fam-class}
outputs potential classes of complete families. 
The final step of \ALG{fam} consists of verifying with \citep[Algorithm 2]{nls-bp},
which of these classes form complete irreducible families. 
\end{proof}

\section{Computing lines and conics on a surface}
\label{sec:exm-alg}

In \SEC{comp-class} we computed some lines and a family of conics on $X$ from
a given birational map $\Mdashrow{\McalH}{\MbbP^2}{X}$ as defined at \EQN{H}.
In this section we would like to compute \emph{all} lines and families of conics 
of $X$.

For computing all lines contained in $X$,
we call \ALG{fam} with input $\McalH$, $\alpha=1$, $\nu=1$ and $\rho=0$.
In the first step of the algorithm we compute $N(X)$ as discussed in \SEC{comp-ns}.
Since the real structure $\Mrow{\sigma}{X}{X}$ acts as the identity on $N(X)$,
all lines will be real.
In the second step we call \ALG{fam-class} with 
$\alpha=1$, $\beta\in [-2, -1]$ so that the union of outputs is 
\[
\McalA=\Mset{ 
\Me_i,~~ \Me_0-\Me_i-\Me_j,~~ 2\Me_0-\Me_1-\ldots-\Me_5 
~}{~ 
0<i<j\leq 5}.
\]
In the third step we verify which of these classes 
correspond to complete irreducible families (see \SEC{comp-class}).
Thus we verify whether the general curve in the linear series of a class is irreducible.
The output of \ALG{fam} is:
\[
\McalB = \{ 
\Me_1,~~
\Me_2,~~
\Me_3,~~
\Me_4,\quad
\Me_0-\Me_1-\Me_4,\quad 
\Me_0-\Me_2-\Me_4,\quad 
\Me_0-\Me_3-\Me_4,\quad 
\Me_0-\Me_4-\Me_5 
\}.
\]
It follows that $X$ contains $8$ lines.
A line with class $\Me_i$ for $i>0$ is not reachable by the birational map $\McalH$.
See \citep[Section 4.1]{nls-bp} for parametrizing such unreachable curves.
For the remaining classes we can compute the linear series in $\MbbP^2$ so
that the images via $\McalH$ of curves in this linear series are lines in $X$.

For computing all conics in $X$ we call \ALG{fam} with 
$\McalH$, $\alpha=2$, $\nu=2$ and $\rho=0$. 
In this case $\McalA=\Mset{\Me_0-\Me_i, 2\Me_0-\Me_i-\Me_j-\Me_k }{ 0<i<j<k\leq 5 }$
and the output of \ALG{fam} is 
\begin{center}
$\McalB$ $=$ 
$\{$ 
$\Me_0-\Me_1$,\quad  
$\Me_0-\Me_2$,\quad  
$\Me_0-\Me_3$,\quad  
$\Me_0-\Me_4$,\quad  
$2\Me_0-\Me_1-\Me_2-\Me_4-\Me_5$,\quad  
$2\Me_0-\Me_1-\Me_3-\Me_4-\Me_5$,\quad 
$2\Me_0-\Me_2-\Me_3-\Me_4-\Me_5$  
$\}$. 
\end{center}
In \SEC{comp-class} we parametrized conics with class $\Me_0-\Me_4$. 
The remaining families of conics can similarly be represented by a parametrization.

\section{Computing circles on surfaces}
\label{sec:circles}

In order to compute circles on rational surfaces we consider the 
\Mdef{M\"obius sphere}
\[
\MbbS^n:=\Mset{ x\in\MbbP^{n+1} }{-x_0^2+x_1^2+\ldots+x_{n+1}^2=0}.
\]
A \Mdef{circle} in $\MbbS^n$ is defined as a real irreducible conic.
The M\"obius sphere is topologically the projective closure of the one-point-compactification $S^n$
of $\MbbR^n$, \Mst circles and lines in $\MbbR^n$ correspond to circles in $\MbbS^n$.
The \Mdef{stereographic projection} with center $(1:0:\ldots:0:1)\in\MbbS^n$ is defined as
\begin{equation*}
\begin{array}{rcl}
\MbbS^n              & \Mdasharrow{}{\McalP}{} & \MbbP^n \\
(x_0:\ldots:x_{n+1}) & \longmapsto             & 
(x_0-x_{n+1}:x_1:\ldots:x_n),
\end{array}
\end{equation*}
with inverse
\begin{equation*}
\begin{array}{rcl}
\MbbP^n          & \Mdasharrow{}{\McalP^{-1}}{} & \MbbS^n \\
(y_0:\ldots:y_n) & \longmapsto                  & 
(
y_0^2+\Delta:
2y_0y_1:
\ldots:
2y_0y_n:
-y_0^2+\Delta
),
\end{array}
\end{equation*}
where $\Delta:=y_1^2+\ldots+y_n^2$.
We can recover $\MbbR^n$ from $\MbbS^n$ as the affine chart 
$\MbbR^n\cong\Mset{x\in \MbbP^n=\overline{\McalP(\MbbS^n)}}{ x_0\neq 0}$.
Notice that circles in $\MbbS^n$ that pass through the center of stereographic projection 
correspond to lines in $\MbbP^n$.

\begin{example}
{\bf(Roman surface)}
\label{exm:roman}
\\
We consider the Roman surface (see \FIG{1}) with parametrization
\begin{equation*}
\begin{array}{rcl}
\MbbP^2       & \Mdasharrow{}{\McalR}{} & X\subset\MbbP^3 \\
(x_0:x_1:x_2) & \longmapsto             & 
(x_0^2+x_1^2+x_2^2 : -x_0x_1 : -x_1x_2 : x_0x_1).
\end{array}
\end{equation*}
The inverse stereographic projection of the roman surface $X\subset\MbbP^3$ into $\MbbS^3$
has parametrization
\begin{equation*}
{
\begin{array}{rcl}
\MbbP^2       & \Mdasharrow{}{\McalP^{-1}\circ\McalR}{} & Y\subset\MbbS^3 \\
(x_0:x_1:x_2) & \longmapsto             & 
\Bigl( x_1^4 + 3x_1^2x_2^2 + x_2^4 + 3x_1^2x_0^2 + 3x_2^2x_0^2 + x_0^4:\\
&& -2 x_0  x_1  (x_1^2 + x_2^2 + x_0^2) : -2 x_2  x_1  (x_1^2 + x_2^2 + x_0^2) :\\ 
&& 2  x_0  x_2  (x_1^2 + x_2^2 + x_0^2) : \\
&& -(x_1^4 + x_1^2x_2^2 + x_2^4 + x_1^2x_0^2 + x_2^2x_0^2 + x_0^4) \Bigr).
\end{array}
}
\end{equation*}
We expect a finite number of circles on the Roman surface $X\subset\MbbP^3$ and 
thus we assume that $\nu-1=0$. 
We call \ALG{fam} with $\McalP^{-1}\circ\McalR$ and $(\alpha,\nu,\rho)=(2,1,0)$.

In step 1 of \ALG{fam} we perform with \citep[Algorithm 1]{nls-bp} 
a basepoint analysis on the map $\McalP^{-1}\circ\McalR$ and find 
the following simple basepoints in $\MbbP^2$
where $\eta^2-\eta+1=0$: 
\begin{center}
$
\begin{array}{ll}
p_1:=(1:\eta - 1: -\eta    ), & p_2:=(1:-\eta: \eta - 1    ), \\
p_3:=(1:-\eta + 1: \eta    ), & p_4:=(1:\eta: -\eta + 1    ), \\
p_5:=(1:\eta - 1: \eta     ), & p_6:=(1:-\eta: -\eta + 1   ),  \\
p_7:=(1:-\eta + 1: -\eta   ), & p_8:=(1:\eta: \eta - 1     ).  
\end{array}
$
\end{center}
Notice that $p_i$ is complex conjugate to $p_{i+1}$ for $i\in\{1,3,5,7\}$.
We denote the exceptional divisor, that results from blowing up the 
projective plane with center $p_i$, by $\Me_i$ for $1\leq i\leq 8$.
The Neron-Severi lattice of $Y\subset\MbbS^3$ is generated by $N(Y)=\Mmod{\Me_0,\Me_1,\ldots,\Me_8}$
with $\sigma_*(\Me_0)=\Me_0$ and $\sigma_*(\Me_i)=\Me_{i+1}$ for $1\leq i\leq 7$.
The class of hyperplane sections of $Y$ equals $\Mh=4\Me_0-\Me_1-\ldots-\Me_8$.

In step 2 of \ALG{fam} we call \ALG{fam-class} with $\Mh\in N(Y)$ and $(\alpha,\beta)=(2,-2)$.
The output is $\Mset{\Mh-2\Me_0+\Me_i+\Me_j}{1\leq i<j\leq 8}$. 
Next we call \ALG{fam-class} with $\Mh\in N(Y)$ and $(\alpha,\beta)=(2,-1)$.
In this case, the output is $\Mset{\Me_0-\Me_i-\Me_j}{1\leq i<j\leq 8}$.

In step 3 of \ALG{fam} we compute the linear
series of $2\Me_0-\Me_1-\ldots-\Me_6$. This linear
series consists of a single curve 
$C:=\Mset{x\in\MbbP^2}{ x_0^2+x_1^2+x^2=0 }$.
However, we detect that $p_i\in C$ for $1\leq i\leq 8$ and
thus the linear series has unassigned basepoints $p_7$ and $p_8$.
Therefore the class of the linear series is $2\Me_0-\Me_1-\ldots-\Me_8$,
which is not the class of a circle.
Since we want to compute the real circles we require that 
the classes are fixed under the involution. 
It follows that the remaining candidate classes for circles 
are $\Mset{\Me_0-\Me_i-\Me_{i+1}}{i\in\{1,3,5,7\}}$.
The linear series $L_{12}$ with class $\Me_0-\Me_1-\Me_2$
is generated by the tuple $(x_0+x_1+x_2)$
so that the the parametrization of $\Mset{x\in\MbbP^2}{x_0+x_1+x_2=0}$
composed with the map $\Mrow{\McalP}{\MbbP^2}{X}$, is the 
parametrization of a circle in $X$.
Similarly, we obtain for the remaining classes 
the linear series $L_{34}=(x_0-x_1-x_2)$, 
$L_{56}=(x_0+x_1-x_2)$ and $L_{78}=(x_0-x_1+x_2)$.
\Mend                                  
\end{example}

\begin{remark}
{\bf(circles versus lines)}
\\ 
Suppose that $\Mdashrow{\McalH}{\MbbP^2}{X\subset\MbbP^4}$ is as defined at \EQN{H}.
In order to compute circles in $X$ we compute with \ALG{fam} complete irreducible families 
of conics in the surface with parametrization 
$\Mdashrow{\McalP^{-1}\circ\McalH}{\MbbP^2}{Y\subset \MbbS^4}$.
We expect at least 8 circles in $Y$, which are the image of $\McalP^{-1}$
of the 8 real lines in~$X$ we computed in \SEC{exm-alg}. The circles in~$Y$, that do not pass through
the center~$(1:0:0:0:0:1)$ of the stereographic projection $\Mdashrow{\McalP}{\MbbS^4}{\MbbP^4}$, 
correspond to circles in~$X$.
\Mend
\end{remark}

\begin{remark}
{\bf(alternative approach)}
\\ 
In case \citep[Algorithm 1]{nls-bp} does not terminate there is an alternative method to compute circles. 
Suppose that $\Mdashrow{\McalH}{\MbbP^2}{X\subset\MbbP^4}$ is as defined at \EQN{H}.
We consider the finite set $B:=\McalH^{-1}(X\cap A)\subset\MbbP^2$ 
where $A:=\Mset{x\in\MbbP^4}{x_0=x_1^2+x_2^2+x_3^2+x_4^2}$ is the \Mdef{Euclidean absolute}.
The real irreducible conics in $\MbbP^4$ that meet $A$ in two complex conjugate points correspond to circles. 
Thus if $C\subset X$ is a real irreducible conic and $\# \McalH^{-1}(C)\cap B=2$ then $C$ is a circle.
For example, we compute with \citep[Algorithm 2]{nls-bp} the curves in the linear series $L_4$ of \SEC{comp-class} that pass through
two given complex conjugate points of $B$.
\Mend
\end{remark}

\section{Acknowledgements}

I thank Josef Schicho for interesting comments. \FIG{1} was created with 
\citep[Surfex]{surfex} and \citep[Povray]{povray}.
We used \citep[Sage]{sage} for the computations. 
Our implementations are open source \cite{linear_series,ns_lattice}.

\bibliography{geometry}

\paragraph{address of author:}
Johann Radon Institute for Computational and Applied 
Mathematics (RICAM), Austrian Academy of Sciences
\\
\textbf{email:} niels.lubbes@gmail.com

\end{document}